\documentclass[12pt]{article}

\RequirePackage{kpfonts}
\RequirePackage[sf,bf,small,raggedright]{titlesec}
\usepackage{dsfont}
\usepackage{comment}
\usepackage{textcomp}
\usepackage{amssymb}
\usepackage{bbm}
\usepackage{fancyhdr}
\usepackage{amsmath}
\usepackage{amsbsy,amsthm}
\usepackage{amscd}
\usepackage{latexsym}
\usepackage{graphicx}   
\usepackage{pdfsync}
\usepackage{blkarray}
\usepackage{multirow}
\usepackage{hyperref}
\usepackage{xcolor}
\usepackage{enumerate}

\usepackage{natbib}
\usepackage{tcolorbox}

\usepackage{tikz}

\usepackage{subcaption}
\usepackage{algorithmic}
\usepackage[justification=centering]{caption}

\usepackage{natbib}
 \def\newblock{\ }%
 \bibpunct[, ]{[}{]}{,}{n}{}{,}%

\usepackage{float}
\usepackage{pgf}
\usetikzlibrary{arrows}

\newcommand{\N}{\mathbb{N}}
\newcommand{\Z}{\mathbb{Z}}
\newcommand{\E}{\mathbb{E}}

\newcommand{\procond}[2]{\mathbb{P}\mathopen{}\left[#1\middle|#2\right]}

\newtheorem{theorem}{Theorem}
\newtheorem{lemma}[theorem]{Lemma}
\newtheorem{corollary}[theorem]{Corollary}
\newtheorem{proposition}[theorem]{Proposition}
\newtheorem{definition}[theorem]{Definition}

\allowdisplaybreaks
\newtheorem{remark}{Remark}

\numberwithin{equation}{section}

\def\IND{\mathbbm{1}}

\newcommand{\argmin}{\mathop{\mathrm{argmin}}}

\newcommand{\V}{\mathbb{V}}
\newcommand{\PP}{\mathbb{P}}
\newcommand{\prob}[1]{\mathbb{P}\left( #1 \right)}

\newcommand{\Lset}{\mathcal{L}} 
\newcommand{\leftmost}{L}
\newcommand{\rightmost}{R}
\newcommand{\leaves}{\Lambda}

\begin{document}

\title{Subtractive random forests with two choices
 \author{
  Francisco Calvillo  \\
LPSM, Sorbonne Université, 4 Place Jussieu, 75005 Paris, France
  \and
  Luc Devroye \\
  School of Computer Science, McGill University, \\
  Montreal, Canada
\and
G\'abor Lugosi \thanks{Corresponding author; email: {\tt gabor.lugosi@gmail.com}}\\
Department of Economics and Business, \\
Pompeu  Fabra University, Barcelona, Spain \\
ICREA, Pg. Lluís Companys 23, 08010 Barcelona, Spain \\
Barcelona Graduate School of Economics
}
}

\maketitle

\begin{abstract}
  Recommendation systems are pivotal in aiding users amid vast online
  content. Broutin, Devroye, Lugosi, and Oliveira proposed
  Subtractive Random Forests (\textsc{surf}), a model that emphasizes
  temporal user preferences. Expanding on \textsc{surf}, we introduce
  a model for a multi-choice recommendation system, enabling users to select from
  two independent suggestions based on past interactions. We evaluate
  its effectiveness and robustness across diverse scenarios,
  incorporating heavy-tailed distributions for time delays. By
  analyzing user topic evolution, we assess the system's
  consistency. Our study offers insights into the performance and
  potential enhancements of multi-choice recommendation systems in
  practical settings.
\end{abstract}


\section{Introduction}

Today's online platforms have a record number of items available for their users. In many cases, presenting them as a catalogue is no longer a good option and can make it difficult for the user to find what they are interested in.
Recommendation systems have become an essential tool to cope with the overload of information available on the web.

A well-known way to approach recommendation systems today is through
deep learning, and many of the most effective recommendation systems
use this principle, see \cite{deeplearning}. However, the mechanism of
these systems is notable for its opacity. In the realm of online
sequential recommendation systems, Broutin, Devroye, Lugosi, and
Oliveira \cite{surf} present an approach, offering a
framework that considers the temporal aspect of user preferences by a
simple mechanism.  They define a model recommending topics based on a
random time delay. The topic that is recommended at time $n$ is the
same that was recommended at time $n-Z_n$, where $(Z_k)_{k\ge 1}$ is a
sequence of i.i.d.\ random variables identically distributed on
$\{1,2,\ldots\}$. This approach led them to study a family of random
forests called subtractive random forests (\textsc{surf}), allowing a
detailed structural study. However, one expects a recommendation
system to make several recommendations at a time, not just one.

To model this more realistic scenario, our paper proposes a two-choice
recommendation system inspired by the same idea. We envision a
scenario where users are presented with two independent
recommendations, drawn from the same temporal recommendation mechanism
and allowing users to select the most appealing option.
The resulting model has some intriguing mathematical properties, and the main goal of this paper
is to analyze the model in order to understand the long-term behavior of such recommendation systems.

\textcolor{black}{While our model provides a mathematically tractable framework to study sequential recommendation dynamics, it is intentionally abstract and does not incorporate key aspects of modern recommender systems such as content information, user-specific features, or contextual data. This abstraction limits its direct applicability to real-world platforms, where personalization and diversity are central goals. However, we view this simplification as a necessary step: by first understanding the behavior of a minimal, analytically transparent model, we lay the groundwork for tackling more complex and realistic settings in which such additional factors can be systematically incorporated.}

\textcolor{black}{Beyond serving as a recommender system in its own right, our model can also be viewed as a theoretical tool to gain insight into more complex algorithms. In particular, it offers a simple framework to reason about phenomena such as the repeated occurrence of topics in classical recommender systems, where machine learning mechanisms are often treated as black boxes.
By studying this simplified yet transparent model, we can uncover structural patterns that would otherwise remain hidden, offering insight into the dynamics of more sophisticated recommendation methods.}

\subsection{A two-choice recommendation model}

We assume that the initial pool of topics is infinite and represented by the set of non-positive integers $\{0, -1, -2,\ldots\}$.
We now consider independent and identically distributed random variables $Z$, $W$, $(Z_n)_{n\ge 1}$ and $(W_n)_{n\ge 1}$
on the set of positive integers $\N$.
Define 
$$
q_i=\PP(Z=i) \quad \text{and} \quad p_i=\PP(Z\ge i), \quad i \ge 1~.
$$

Each topic $i \leq 0$ is assigned a preference value $U_i$ within the range $[0,1]$, where we assume that $(U_n)_{n\leq0}$ is a sequence of random variables independent of the sequences $(Z_n)_{n\geq1}$ and $(W_n)_{n\geq1}$. \textcolor{black}{In the context of our model, the term user refers to a representative individual interacting with the system, and the values $U_i$ represent how appealing each topic is to this user.}

Following \cite{surf}, we can define a sequential random colouring of the positive integers as follows. For each non-positive integer $i\le 0$, define its colour $C_i=i$ and its preference value $V_i=U_i$. We assign colours and preference values to the positive integers $n\in \N$ by the recursion 
$$
\left\{
\begin{array}{ll}

    C_n=C_{n-Z_n} \text{ and } V_n=V_{n-Z_n} & \text{ if } V_{n-Z_n}< V_{n-W_n} \\
    C_n=C_{n-W_n} \text{ and } V_n=V_{n-W_n} & \text{ otherwise }~.
\end{array}
\right.
$$
Thus, by identifying the colour $C_n$ as the topic recommended at time $n\ge 1$, this definition means that at the time instant $n\ge 1$, the user receives two recommendations (the topic $C_{n-Z_n}$, and the topic $C_{n-W_n}$), and chooses the one with the lowest (best) preference value.

This process naturally defines a random directed graph whose vertex set is $\Z$ by drawing an edge from any positive integer $n\ge 1$ to any integer $m<n$ if and only if $m=n-Z_n$ or $m=n-W_n$. Vertices with negative index are called \textit{leaves}, as is customary for final nodes in dags (i.e., directed acyclic graphs). 
Let $T_n$ denote the set of vertices that are reached from the vertex $n\ge 1$. This set can be defined recursively by saying that $T_n=\{ n \}$ for any $n\le 0$, and $$T_n=\{n\}\cup T_{n-Z_n} \cup T_{n-W_n}$$ for any $n\ge 1$. Define $$\Lset_n=T_n\cap \{0, -1, -2, \dots\}$$ as the set of vertices with a non-positive index that is reached from $n$. In other words, $\Lset_n$ is the set of leaves in the subgraph of vertices reached from $n$. Note that 
\begin{equation*}
    C_n=\argmin_{i\in \Lset_n} U_i ~\text{ and }~ V_n=\min_{i\in \Lset_n} U_i~.
\end{equation*}

One can assess the system's long-term performance by analyzing the asymptotic behavior of the sequence $(V_n)_{n\ge 1}$. If $V_n$ approaches $\inf_{i\le 0} U_i$ as $n\to \infty$, then this means that after waiting a sufficient amount of time, the topics recommended to the user tend to align more closely with their preferences. This gives us a consistency criterion for our model. 
We consider the three following configurations of the preference values: 

\begin{enumerate}
    \item[(i)]  $U_0<U_{-1}<U_{-2}< U_{-3}\dots$
    \item[(ii)] $U_0>U_{-1}>U_{-2}> U_{-3}\dots$ 
    \item[(iii)] $(U_n)_{n\le 0}$ i.i.d.\ and uniformly distributed in $[0,1].$
\end{enumerate}

These assumptions correspond to natural scenarios and allow us to study the long-term consistency of the recommendation system.
In the first case, topics have a monotone preference with the most recent (corresponding to $i=0$) being the most preferred one.
In the second case, older topics are preferred, while in the third case, topics have a random order of preference. 
We say that the system is consistent if, in the long run, it offers near-optimal recommendations to the user in terms of their preference.
In Section \ref{sec: results}, we rigorously define various notions of consistency, corresponding to the scenarios above. 
These definitions lead us to study the sequence $(V_n)_{n\ge 1}$ by studying the set of leaves $\Lset_n$.
We will pay particular attention to the three following random variables:

\medskip

\noindent $\bullet$ $\rightmost_n=\max \left\{ i : i \in  \Lset_n \right\}$ (the rightmost leaf in $\Lset_n$)

\noindent $\bullet$ $\leftmost_n = \min \left\{ i : i \in  \Lset_n \right\}$ (the leftmost leaf in $\Lset_n$)

\noindent $\bullet$ $\leaves_n = |\Lset_n|$ (the number of leaves reached from $n$).

\medskip

Each random variable helps us understand the long-term behavior of $V_n$ in the three initial configurations described above. 
Ideally, one would hope that $V_n\to \inf_{i\le 0} U_i$.
In configuration (i)  this is equivalent to $\rightmost_n \to 0$, 
in configuration (ii) it is equivalent to $\leftmost_n\to -\infty$, while in case (iii) to $\leaves_n\to \infty$.
However, one cannot expect to have $\rightmost_n\to 0$ in a general case. 
To introduce a more reasonable consistency criterion for case (i), we require boundedness of the sequence $(\rightmost_n)_{n\ge 1}~.$

Let us define the subsets $T_n^Z=T_n^W=S_n=\{n\}$ for any $n\le 0$, and let 

\noindent $$T_n^Z=\{n\}\cup T_{n-Z_n}^Z~,$$

\noindent $$T_n^W=\{n\}\cup T_{n-W_n}^W~,$$

\noindent $$\text{and } S_n=\{n\}\cup S_{n-\min(Z_n,W_n)}$$

\noindent for any $n\ge 1$, so that $T_n^Z\subset T_n$ (resp. $T_n^W\subset T_n$) is the set of vertices that can be reached from $n$ by following only the $Z$-edges (resp. $W$-edges), and $S_n\subset T_n$ is the set of vertices that can be reached from $n$ by following only the shortest edges.

\begin{figure}[H]
\begin{center}
\leavevmode
\includegraphics[width=\textwidth]{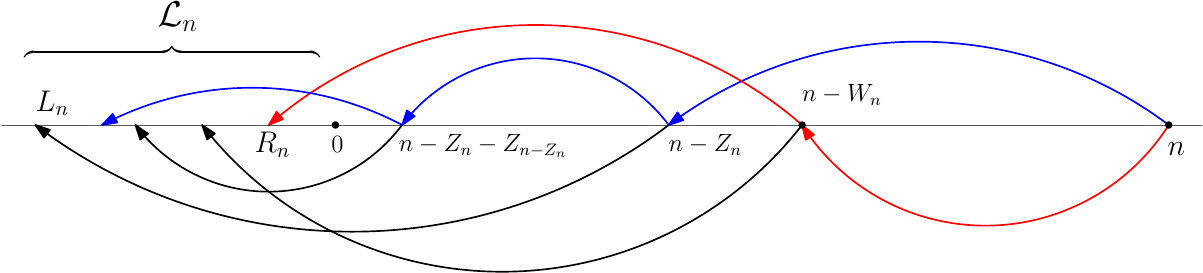}
\end{center}
\caption{Illustration of the subgraph induced by $T_n$. Shown are the definitions of $\leftmost_n$, $\rightmost_n$, and $\Lset_n$. Blue edges indicate the $Z$-chain of vertices $T_n^Z$, while red edges connect all vertices in the chain $S_n$.}
    \label{fig: sous arbres}
\end{figure}

Finally, for any $n\ge 1$, let $r_n=\prob{0\in T_n^Z}=\prob{0\in T_n^W}$ denote the probability that vertex $0$ belongs to the $Z$-chain starting at vertex $n$. We set $r_0=1$ so that $r_n$ satisfies the recursion
\begin{equation}
    r_n=\sum_{i=1}^n q_i r_{n-i}~.\label{r recursion}
\end{equation}
\color{black}
For convenience, a table summarizing the main notation used throughout the paper is provided in the Appendix.
\begin{remark}
Observe that the chains $T_n^Z$, $T_n^W$, and $S_n$ can be naturally interpreted as a renewal process, where each step along the chain corresponds to a renewal event, and the edge lengths play the role of renewal intervals. Within this framework, the quantity $r_n$ coincides with the renewal measure of the process corresponding to $T_n^Z$.
\end{remark}
\color{black}


\subsection{Related work}
This paper builds on the work of Broutin, Devroye, Lugosi and Oliveira \cite{surf}, who studied a one-choice version of this model by examining the properties of a graph whose vertex set is $\Z$, and whose edges are $\{n,n-Z_n\}$ for every $n\ge 1$, given some sequence $(Z_n)_{n\ge 1}$ of i.i.d. random variables distributed in $\{1,2,3,\dots\}$. This defines a random forest where each non-positive integer is the root of a tree. They also define a random coloring by setting $C_i=i$ for all $i\le 0$, and $C_n=C_{n-Z_n}$ for all $n\ge 1$. Each tree in the forest corresponds to a colour (i.e., a topic). The graph that we consider in our paper is nothing else than the superposition of two independent copies of this forest, and the sets $T_n^Z$ and $T_n^W$ defined above are the chains that connect the vertex $n$ to the root of its tree in each of these two copies.

One of the main results in \cite{surf} shows that the process has two
distinct behaviors: if $\E Z=\infty$, then, almost surely, all trees
in the forest are finite, meaning that no topic will be recommended
infinitely often, while if $\E Z < \infty$, after some random amount
of time, all vertices connect to the same tree almost surely, meaning
that all topics recommended to the user become the same. 
In the two-choice
model however, there are three distinct regimes: $\E Z <\infty$; $\E Z= \infty$ yet
$\E \min(Z,W)<\infty$; and $\E \min(Z, W) = \infty$.
We call these the light-, moderate-, and heavy-tailed regimes.

The random forest in \cite{surf} appears as a subgraph of a random graph model previously studied by Hammond and Sheffield in \cite{hammondsheffield2013}. Given a sequence $(Z_n)_{n\in \Z}$ of independent random variables, identically distributed in $\{1,2,3,\dots\}$, they consider a graph with vertex set $\Z$ such that vertices $n,m\in \Z$, with $m<n$, are connected by an edge if and only if $m=n-Z_n$. One can obtain the random forest in \cite{surf} by removing all edges between any two vertices with a non-positive index. They show that the random graph has almost surely a unique component when the sum $\sum_{n=1}^\infty r_n^2$ converges, and the graph almost surely has infinitely many connected components when the sum diverges.

\color{black}
Other variants of the (single-choice) random subtractive process have been analyzed in related probabilistic models. Blath, González, Kurt, and Spano \cite{BlGoKuSp13} studied long-range seed bank models, where the genealogical structure of populations is governed by random time jumps analogous to the subtractive mechanism. Chierichetti, Kumar, and Tomkins \cite{ChKuTo20} examined sequence models whose asymptotic behavior shares similarities with renewal-type processes underlying our setting. Baccelli, Haji-Mirsadeghi, Khezeli, and Sodre \cite{BaSo19, BaHaKh18, BaHaKh22} explored subtractive dynamics in the context of unimodular random graphs and population processes. Finally, Igelbrink and Wakolbinger \cite{IgWa22} investigated reinforced urn schemes where coalescence probabilities play a role closely related to recurrence properties of subtractive chains.
\color{black}

\subsection{Results}\label{sec: results}

The goal of this paper is to understand the behavior of the two-choice
model by studying its consistency in each of the three configurations
(i),(ii), and (iii), for three types of tails, light, moderate, and
heavy:
\begin{enumerate}
        \item $\E Z < \infty$ (light tails)
        \item $\E Z = \infty$ and $\E \min(Z,W) < \infty$ (moderate tails)
        \item $\E \min(Z,W) = \infty$ (heavy tails)
\end{enumerate}

We now introduce the main definitions of consistency.

\begin{definition}
We say that the system is consistent in configuration (ii) if $\leftmost_n\to - \infty$; 
while it is consistent in configuration (iii) if $\leaves_n\to \infty$. 
We differentiate between strong and weak consistency based on whether
convergence occurs almost surely or in probability.
In the case of configuration (i), we say that the system is weakly consistent if $(\rightmost_n)_{n\ge 1}$ 
is a tight sequence. It is strongly consistent if the sequence $(\rightmost_n)_{n\ge 1}$ 
is almost surely bounded.
\end{definition}

\begin{remark}\label{pareto}
\textsc{Pareto tails.}
Consider a distribution with Pareto tails, that is, $q_n = \Theta (1/n^{1+\alpha})$, with $\alpha > 0$. This implies that $p_n = \Theta(1/n^\alpha)$. When $\alpha > 1$, we see that $\E Z < \infty$.
The more interesting situation is when $\alpha \in (0,1]$. For $\alpha \in (1/2,1]$, we have $\E Z = \infty$, yet $\E \min (Z,W) < \infty$.
Finally, for $\alpha \in (0,1/2]$, we have $\E \min(Z,W) = \infty$.
For light-tailed $Z$, we recall from Broutin, Devroye, Lugosi, Oliveira \cite{surf} that $r_n \to 1/\E Z$. For moderate and heavy tails,
however, we have $r_n \to 0$. When $\alpha \in (0,1)$, then $Z$ is in the domain
of attraction of the extremal stable law with parameter $\alpha$, which itself
has a tail distribution function that decays at the rate $1/n^\alpha$ (see, e.g., Ibragimov and Linnik \cite{ibragimov}, Zolotarev \cite{zolotarev} or Malevich \cite{malevich2015}). One can then deduce that $r_n$ decays at the rate 
$1/n^{1-\alpha}$. In particular, $\sum_n r_n^2 < \infty$ when $\alpha \in (0,1/2)$, which roughly corresponds to the case of heavy tails.
\end{remark}

Next, we describe the main results of the paper. 
For a summary, see the table at the end of this section.  
  
\subsubsection*{Light tails}
When $\E Z < \infty$, time jumps are so short that $T_n$ can't reach distant leaves, thus forcing the set $\Lset_n$ to be bounded.
This case does not require much work, since one can quickly determine consistency or inconsistency in each of the three configurations just by looking at the sequence $(\leftmost_n)_{n\ge 1}$, which happens to be bounded almost surely.

\begin{theorem}\label{M bounded as}
Let $\leftmost_\infty=\inf_{n\in \N} \leftmost_n$. If $\E Z < \infty$, then $\leftmost_\infty$ is finite almost surely.
Thus, the sequences $(\rightmost_n)_{n\in\N},(\leftmost_n)_{n\in\N}$ and $(\leaves_n)_{n\in\N}$ are almost surely bounded.
\end{theorem}

\begin{proof}
Note that
\begin{align*}
    \prob{\leftmost_\infty \le -x} \le 2\sum_{n=1}^{\infty}\prob{n-Z_n\le -x}=2\sum_{n=1}^\infty p_{n+x}~.
\end{align*}
Since $\E Z < \infty $, we know that $\sum_{n=1}^\infty p_{n+x}$ is finite and goes to $0$ as $x\rightarrow +\infty$. Thus, by continuity of measure,
\begin{align*}
    \prob{\leftmost_\infty = -\infty} = \lim_{x\rightarrow +\infty} \prob{\leftmost_\infty \le -x}=0~.
\end{align*}
Since $|\leftmost_n|\ge \leaves_n$ and $|\leftmost_n|\ge |\rightmost_n|$, this also implies that $(\leaves_n)_{n\in \N}$ and $(\rightmost_n)_{n\in\N}$ are almost surely bounded.
\end{proof}

\subsubsection*{Moderate and heavy tails}
When $\E \min (Z,W) = \infty$, a different behavior emerges:

\begin{theorem}\label{a.s. thm Emin=infty}
    Let $Z$, $W$, $(Z_n)_{n\ge 1}$ and $(W_n)_{n\ge 1}$ be i.i.d.\
    random variables. Assume furthermore that $\E \min (Z,W) =
    \infty~.$ Then, with probability one,
  $$
  \liminf_{n\to \infty} \leaves_n < \infty~.
  $$
In other words, the system is not strongly consistent in the
heavy-tailed regime in configuration (iii).
  \end{theorem}

\begin{proof}
    Note that if $\min (Z_n, W_n) \ge n$, then $\leaves_n\in \{1,2\}$. Thus,
    to prove Theorem \ref{a.s. thm Emin=infty}, it suffices to show
    that with probability one, $\min(Z_n,W_n) \ge n$ happens
    infinitely often. This follows by the second Borel-Cantelli lemma since the events 
    $$(\{\min(Z_n, W_n)\ge n\})_{n\ge 0}$$
    are independent and 
    $$\sum_{n\ge 0} \prob{\min(Z_n, W_n)\ge n} = \E \left( \min (Z, W) \right)=\infty~.$$
\end{proof}

The behavior of the set $\Lset_n$ for moderate and heavy tails is
better understood by examining its extreme points $\leftmost_n$ and $\rightmost_n$.
The following two theorems are proved in Section \ref{sec:rlleaves}.

\begin{theorem} \label{M diverges a.s. }
    If $q_i>0$ for all $i\in \N$ and $\E Z=\infty$, then, with
    probability one, $$\leftmost_n\to -\infty~.$$
Thus, in configuration (ii) the system is strongly consistent in the
moderate and heavy-tailed regimes.
    \end{theorem}

Strong (in-)consistency in configuration (i) is established in the
next result:

\begin{theorem}\label{a.s thm Emin<infty}
    Let $Z$, $W$, $(Z_n)_{n\ge 1}$ and $(W_n)_{n\ge 1}$ be i.i.d.\ random variables. 
    \begin{enumerate}
        \item If $\E \min(Z,W)<\infty$, then $\sup_{n\in\N} |\rightmost_n| < \infty$  almost surely.
        \item If $\E \min(Z,W) = \infty$, then $\sup_{n\in\N} |\rightmost_n| = \infty$ almost surely.
    \end{enumerate}
    
\end{theorem}
Given that the leftmost leaf in $\Lset_n$ goes to $-\infty$ and the rightmost leaf remains bounded in the moderate heavy-tail regime (i.e., when $\E Z = \infty$ and $\E \min(Z,W)<\infty$), one can expect that the total number of leaves reached by the vertex $n$ diverges under these two assumptions. Thus, strong consistency is observed in each of the three configurations (i), (ii), and (iii).

On the other hand, only configuration (i) leads to strong consistency
when the distribution of $Z$ has a heavy-tail (i.e., when $\E
\min(Z,W) = \infty$). However, when $\sum_{n=1}^\infty r_n^2 <
\infty$, the model remains weakly consistent in configuration (iii),
as established by the next two results. Theorems \ref{ L to infinity
  almost surely} and \ref{weak thm Emin=infty} are proved in  Sections
\ref{sec:moderate} and \ref{sec:heavy}, respectively.

\color{black}
\begin{theorem}\label{ L to infinity almost surely}
    Assume that the distribution of $Z$ is aperiodic, that is, the greatest common divisor of the support of $Z$ is $1$. If $Z$ exhibits a moderate-sized tail (i.e., $\E \min (Z,W)<\infty$ and $\E Z = \infty$), then 
    $$\leaves_n\to \infty \text{ almost surely. }$$
    \end{theorem}
\color{black}

\begin{theorem}\label{weak thm Emin=infty}
    When $Z$ has a heavy tail (i.e., $\E \min(Z,W) = \infty$), and $\sum_{n\ge0}r_n^2 < \infty~,$ then $$\leaves_n\to \infty \textit{ in probability.}$$
\end{theorem}

We recall from Remark \ref{pareto} that the summability of $r_n^2$ is assured for nearly all
heavy-tailed $Z$.

\subsection{Summary}\label{sec:summary}

The following table summarizes our findings. 
Observe that weak and strong behavior coincide in most cases.
They only differ in the extreme heavy tail case ($\E \min(Z,W) = \infty$). For an optimal user experience, one would need $\leaves_n \to \infty$ almost surely, and the only case that assures this is when we have moderate tails. In addition, such moderate tails also guarantee strong consistency for the decreasing and increasing input scenarios (as is apparent from the strong consistency of $\leftmost_n$ and $\rightmost_n$).
\vspace{.5cm}

\begin{center}

\begin{tabular}{|c||c|c|c|}
        \hline
        & $\E Z<\infty$ & \vtop{\hbox{\strut $\E Z = \infty$ and}\hbox{\strut $\E\min(Z,W)<\infty$}} & $\E\min(Z,W)=\infty$\\ 
        \hline \hline
        &\multicolumn{3}{|c|}{Strong consistency} \\
        \hline
        $\leftmost_n\to -\infty $ a.s. & no (Thm. \ref{M bounded as}) \ & yes (Thm. \ref{M diverges a.s. }) & yes (Thm. \ref{M diverges a.s. })\\ 
        \hline
        $\leaves_n\to \infty $ a.s.& no (Thm. \ref{M bounded as}) & yes (Thm. \ref{ L to infinity almost surely}) & no (Thm. \ref{a.s. thm Emin=infty})\\ 
        \hline
        $\rightmost_n$ bounded a.s. & yes (Thm. \ref{M bounded as}) & yes (Thm. \ref{a.s thm Emin<infty})  & no (Thm. \ref{a.s thm Emin<infty})\\ 
        \hline \hline 
        &\multicolumn{3}{|c|}{Weak consistency} \\
        \hline 
        $\leftmost_n \overset{\PP}{\to} -\infty$ & no (Thm. \ref{M bounded as})  & yes (Thm. \ref{M diverges a.s. }) & yes (Thm. \ref{M diverges a.s. })\\ 
        \hline
        $\leaves_n \overset{\PP}{\to} \infty$ & no (Thm. \ref{M bounded as}) & yes (Prop. \ref{weak behavior case 2}) & yes if $\sum r_n^2 < \infty$ (Thm. \ref{weak thm Emin=infty})\\ 
        \hline
        $\rightmost_n$ is tight & yes (Thm. \ref{M bounded as}) & yes (Thm. \ref{a.s thm Emin<infty}) & \\ 
        \hline
\end{tabular} 
\end{center}

  The rest of the paper contains the proofs of the results stated above.
In Section \ref{sec:onechoice} we recall some properties of the
single-choice model that are useful in our analysis.
In Section \ref{sec:rlleaves}, Theorems \ref{M diverges a.s. } and
\ref{a.s thm Emin<infty} are proven.
The main technical content of the paper is presented in Section
\ref{sec:numberofleaves}
where the number of leaves is examined in both the moderate-, and 
heavy-tailed cases, culminating in the proofs of Theorems \ref{ L to infinity
  almost surely} and \ref{weak thm Emin=infty}.

\subsection{Simulations}
\color{black}
To complement our theoretical analysis, we conducted a series of simulations to examine the empirical behavior of the model. We considered discrete Pareto distributions for the random variables $Z$ and $W$, with probabilities $q_n = c/n^{1+\alpha}$, so that varying $\alpha$ produces light-tailed, moderate-tailed, and heavy-tailed regimes. Our analysis focused on the evolution of the preferential value $V_n$ of the topic recommended at time $n$ in configuration~(iii), where topics are assigned independent uniform preference values.  

Figure~\ref{fig: Vn simulation} presents single realizations of $V_n$ for each tail regime, illustrating the emergence (or absence) of strong consistency. Figure~\ref{fig:weak} shows the corresponding averages over $1000$ independent simulations, which emphasize the form of weak consistency predicted in the heavy-tailed setting. The simulations reveal three qualitatively distinct behaviors:
\begin{itemize}
    \item \textbf{Light-tailed regime ($\alpha = 1.2$):} no evidence of consistency; recommendations stabilize around a single topic, and $V_n$ does not converge to $0$.
    \item \textbf{Moderate-tailed regime ($\alpha = 0.6$):} strong consistency, with the process converging reliably to a dominant topic.
    \item \textbf{Heavy-tailed regime ($\alpha = 0.3$):} strong inconsistency at the trajectory level, but weak consistency visible in the averaged dynamics, in agreement with Remark~\ref{pareto}.
\end{itemize}

\begin{figure}[H]
  \centering
  \begin{minipage}[t]{0.49\textwidth}
    \centering
    \includegraphics[width=\linewidth]{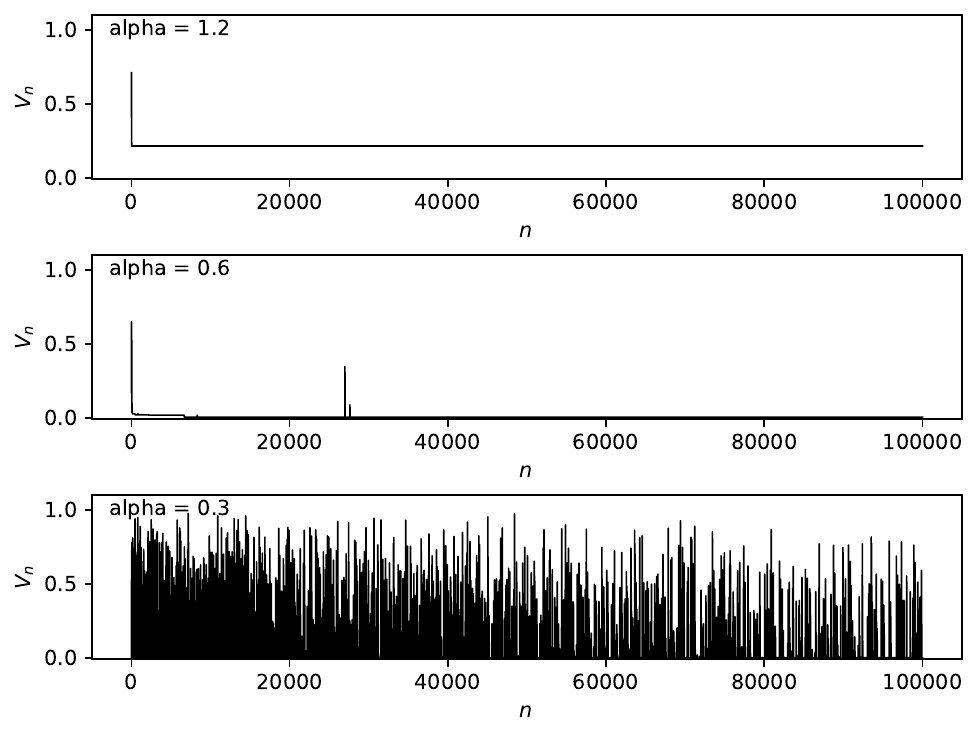}
    \caption{Simulation of the evolution of $V_n$ in configuration (iii) for a single run across light-tailed, moderate-tailed, and heavy-tailed regimes.}
    \label{fig: Vn simulation}
  \end{minipage}
  \hfill
  \begin{minipage}[t]{0.49\textwidth}
    \centering
    \includegraphics[width=\linewidth]{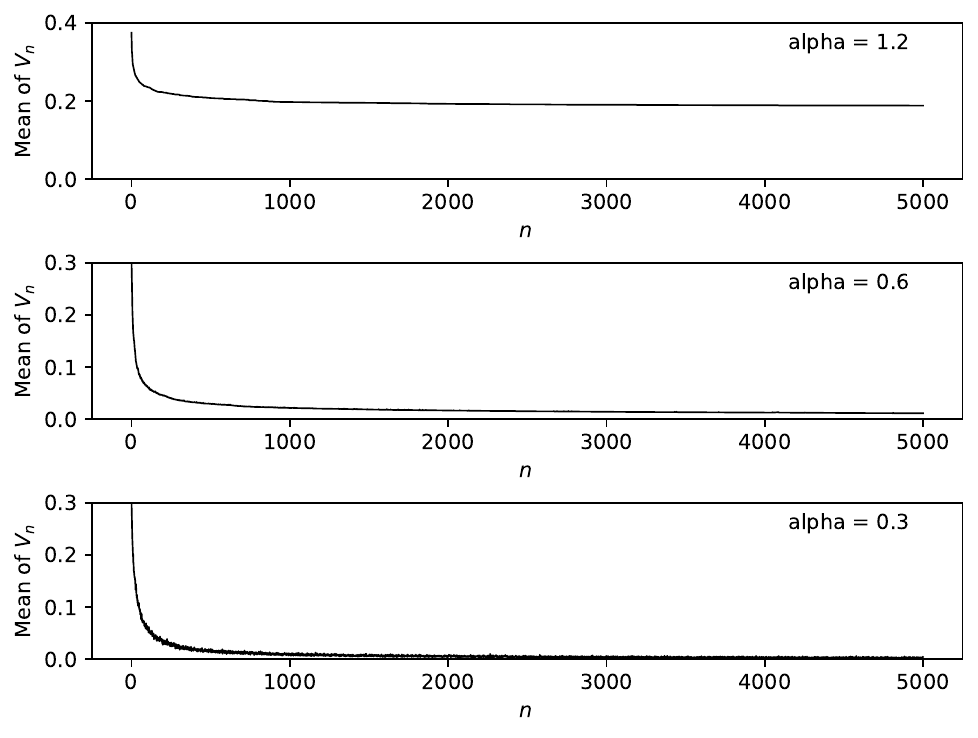}
    \caption{Evolution of the empirical mean of $V_n$ in configuration (iii)$~$over $1000$ independent simulations for light-tailed, moderate-tailed, and heavy-tailed regimes.}
    \label{fig:weak}
  \end{minipage}
\end{figure}
Configurations~$(i)$ and~$(ii)$ were also examined, and the corresponding simulation outcomes are in full agreement with the theoretical results established for these regimes. Since their qualitative behavior is straightforward and consistent with the analytical predictions, the corresponding plots are omitted for brevity.

In addition to tracking the evolution of $V_n$, it is also instructive to visualize the subgraph induced by the vertices $T_n$, consisting of the vertices reached by vertex $n$. Coloring vertices with positive index in blue and vertices with non-positive index (leaves or topics) in red provides a structural perspective on the process. The resulting plots (see Figure~\ref{fig: Tn visualization}) make the differences between the three regimes strikingly clear: in the light-tailed case, the short edges yield a dense graph with many blue vertices but very few leaves; in the moderate-tailed case, the number of leaves grows substantially while blue vertices remain numerous; and in the heavy-tailed case, the dominance of long edges produces a sparse structure with very few blue vertices and a predominance of leaves. These observations are fully consistent with our theoretical analysis and complement the intuition provided by the $V_n$ plots.

\begin{figure}[H]
\centering
\includegraphics[scale=0.45]{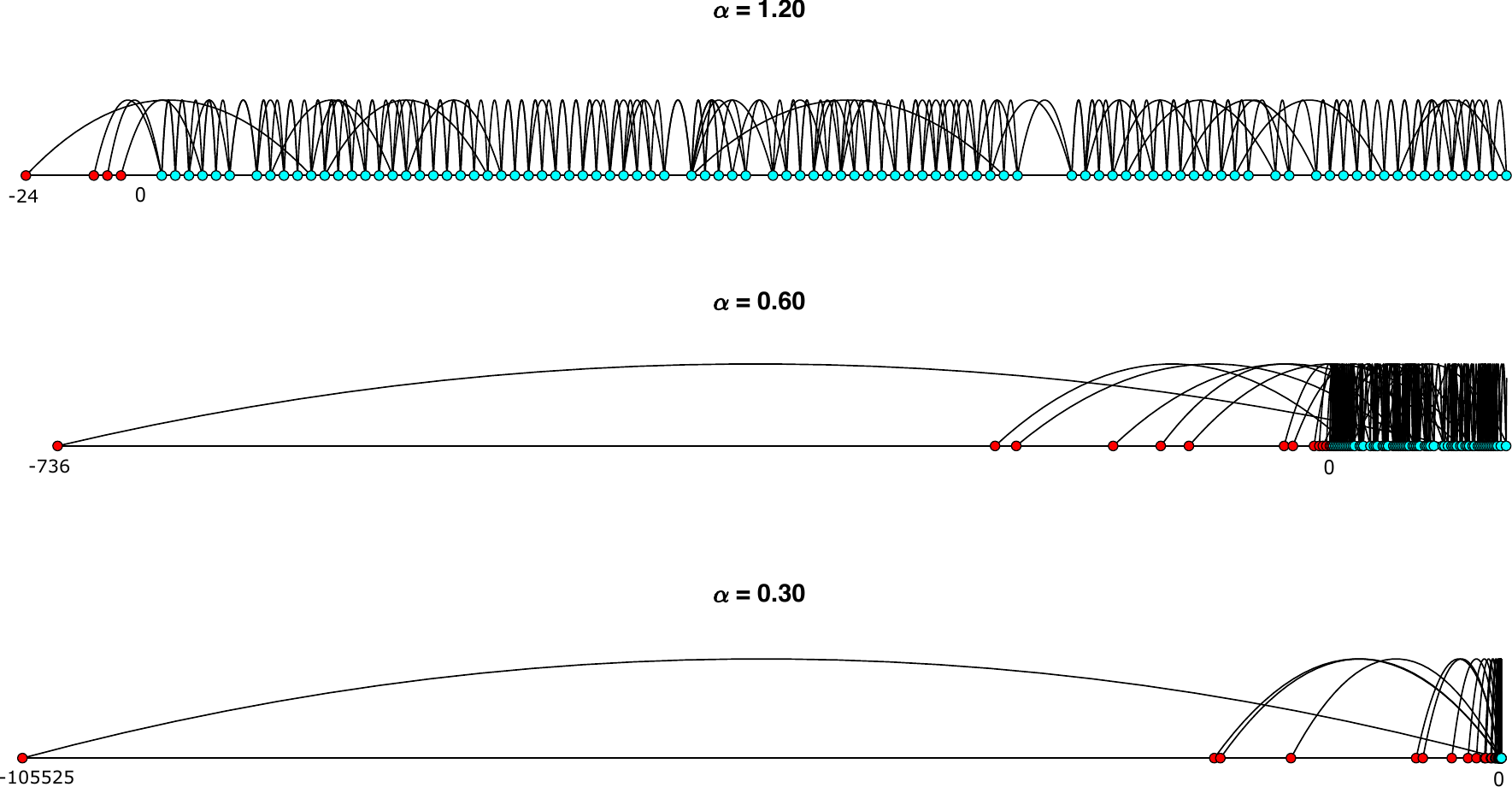}

\caption{Visualization of the subgraph induced by $T_n$ in the light-, moderate-, and heavy-tailed regimes, with $n = 150$.}
    \label{fig: Tn visualization}
\end{figure}
\color{black}

\section{The one-choice model}
\label{sec:onechoice}

Some of our proofs use the observation that the graph contains several copies of the subtractive random forest studied in \cite{surf}. For example, Theorem 1 in \cite{surf} states that there is a unique infinite tree in the forest when $\E Z < \infty$:

\begin{theorem}[Broutin, Devroye, Lugosi, Oliveira \cite{surf}]\label{th1 surf}
Let $Z$ and $(Z_n)_{n\ge 1}$ be i.i.d.\ random variables. Assume $\E Z < \infty$ and $q_1>0~.$ Then there exists a positive
random variable $N$ with $\prob{N<\infty}=1$ such that, with
probability one,
$$N\in T_n^Z  \quad \textit{for all } n\ge N~.$$
\end{theorem}

\begin{remark}
Note that the actual statement of the theorem in \cite{surf} is slightly weaker, as it only asserts the existence of a random variable $N$, finite almost surely, such that for all $n \geq N$, the vertex $n$ has the same colour as vertex $N$. This does not
necessarily imply that $N \in T_n^Z$. Nevertheless, one can add this
property of $N$ without changing the proof provided in \cite{surf}.
\end{remark}

Define $Z'_n = \min(Z_n, W_n)$ and $Z' = \min (Z,W)$ and note that $S_n = T_n^{Z'}$. By Theorem \ref{th1 surf}, we obtain the following result for moderate and light tails.

\begin{corollary}\label{corollary surf}
    Let $Z$, $W$, $(Z_n)_{n\ge 1}$ and $(W_n)_{n\ge 1}$ be i.i.d.\ random variables.
    If $\E \min (Z,W) < \infty$ and $q_1>0$, then there exists a
    positive random variable $N$ with $\prob{N<\infty}=1$ such that,
    with probability one,
 $$N\in S_n \quad \textit{for all } n\ge N~.$$
\end{corollary}

When $\E Z = \infty$, Theorem 3 in \cite{surf} states that every tree in the forest is finite. This translates as follows.

\begin{theorem}[Broutin, Devroye, Lugosi, Oliveira \cite{surf}]\label{th2 surf}
    Let $Z$ and $(Z_n)_{n\ge 1}$ be i.i.d.\ random variables.
    If $\E Z = \infty$ and $q_i > 0$ for all $i \ge 1$, then 
    $$\PP( \cup_{ i \le 0} \left[ \left|\{n\ge 1 : i\in T_n^Z\}\right|=\infty \right] )=0~.$$
\end{theorem}

\section{Rightmost and leftmost leaves}
\label{sec:rlleaves}
  
As stated in Theorem \ref{M diverges a.s. }, the leftmost leaf $\leftmost_n$ reached by vertex $n$ diverges to $-\infty$ whenever $\E Z = \infty$. This can be seen as a consequence of the behavior of the one-choice model, as one only needs to check the divergence to negative infinity of the leaf $\min T_n^Z$ reached by vertex $n$ by following only the $Z$-edges.

\begin{proof}[Proof of Theorem \ref{M diverges a.s. }] From Theorem
  \ref{th2 surf} we know that with probability one, for all $k\le 0$,
  there are at most a finite number of integers $n\ge 1$ such that
  $k\in T_n^Z$. Thus, the sequence $(\min T_n^Z)_{n\in\N}$, taking
  values in $\{0,-1,-2,\dots\}$, cannot take the same value an
  infinite number of times, so it goes to $-\infty$ as $n\to \infty$
  with probability one. This concludes the proof, since $\leftmost_n\le \min T_n^Z$ for all $n\ge 1$.
    
\end{proof}

\begin{proof}[Proof of Theorem \ref{a.s thm Emin<infty}]
    Assume that $\E \min(Z,W) < \infty$. From Corollary \ref{corollary surf}, we know that there exists some random variable $N$ with $\prob{N<\infty}=1$ such that, with probability one, $N\in S_n$ for all $n\ge N$.
    In particular, for any $n\ge N$ we have $T_N\subset T_n$, which implies that $|\rightmost_N|\ge |\rightmost_n|$. Thus, with probability one, we have $$\sup_{n\in \N}|\rightmost_n| \le \max_{1\le k \le N} |\rightmost_k| < \infty~.$$
When $\E \min(Z,W) = \infty$, we have  
$$\sum_{n\ge 1} \prob{\min (Z_n,W_n)\ge 2n}=\sum_{n\ge 1}p_{2n}^2 =\infty~.$$
    Hence, by the second Borel-Cantelli lemma, we know that with probability one we have $\min(Z_n,W_n)\ge 2n$ infinitely often, which implies that $\rightmost_n\le -n$ infinitely often and proves the second statement of the theorem.
    
\end{proof}

\section{Number of leaves}
\label{sec:numberofleaves}

\subsection{Moderate tails}
\label{sec:moderate}

Consider the number $\leaves_n$ of leaves reached by vertex $n$, when the distribution of $Z$ is such that $\E Z = \infty$ and $\E \min(Z,W) < \infty$. 
We first show divergence in probability,
implying weak consistency in configuration (iii).


\color{black}
\begin{proposition}\label{weak behavior case 2}
    Assume that the distribution of $Z$ is aperiodic. If $\E Z = \infty$ and $\E \min(Z,W) < \infty$, then
    $$
    \leaves_n \to \infty \quad \text{in probability as } n \to \infty.
    $$
\end{proposition}

The proof uses the balance that arises in the moderate-tail regime. The distribution of delays is sufficiently heavy to ensure that each non-leaf vertex in $T_n$ has a non-negligible chance of reaching a leaf, but light enough to guarantee that $T_n$ contains a substantial number of non-leaf vertices. Consequently, the overall number of opportunities to generate new leaves grows significantly with $n$. This balance is illustrated in Figure~\ref{fig: Tn visualization}: only in the moderate tail regime ($\alpha = 0.6$ in the figure) do we simultaneously observe a sizable number of blue vertices (non-leaves) and a large population of red vertices (leaves).  

To formalize this heuristic, we focus on the chain $S_n$, obtained by repeatedly following the shortest outgoing edge starting from vertex $n$. In the moderate-tailed regime, this chain typically contains many vertices, and each of these vertices has an additional opportunity to connect to a leaf through its longest outgoing edge. We therefore define
\[
J_n = \sum_{m=1}^n \IND_{m \in S_n}\,\IND_{\max(Z_m, W_m) \ge m},
\]
so that $J_n$ counts the number of vertices in $S_n$ that are connected to a leaf via their longest outgoing edge. Controlling the growth of $J_n$ makes it possible to turn the heuristic balance into a rigorous statement.

\color{black}

\begin{lemma}\label{Divergence}
The expected number of vertices in the chain $S_n$ that connect to a leaf through their longest outgoing edge grows to infinity with $n$:
\[
\E[J_n] \;\to\; \infty \quad \text{as } n \to \infty~.
\]
\end{lemma}

\begin{proof}  
Note that for $n\ge m\ge 1$, the events $\{m\in S_n\}$ and $\{ \max(Z_m,W_m) \ge m\}$ are independent. Define $v_m := \PP(0\in S_m)$ for all $m\ge 1$ and let $v_0=1$, so that $$\E J_n = \sum_{m=1}^n\PP(m\in S_n) \PP(\max(Z_m,W_m) \ge m) = \sum_{m=1}^n v_{n-m}\PP(\max(Z_m,W_m) \ge m)~. $$

\color{black}
\noindent The sequence $(v_n)_{n\ge 0}$ can be better understood by considering the arithmetic renewal process $M_n$ with holding times given by the i.i.d.\ random variables $(\min(Z_i,W_i))_{i\ge 1}$, and initialized with $M_0 = 0$. In other words,  
\[
M_n = \sum_{i=1}^n \min(Z_i,W_i)~.
\]
By definition, $v_n$ is the probability that $n$ can be expressed as a sum of independent random variables $(\min(Z_i,W_i))_{i\ge 0}$. In terms of the renewal process, this means that
\[
v_n = \PP(\exists k \ge 0 : M_k = n)~.
\]
Hence, the sequence $(v_n)_{n \ge 0}$ corresponds to the renewal measure of the aperiodic renewal process $(M_n)_{n \ge 0}$. By the Erd\H{o}s–Feller–Pollard renewal theorem \cite{erdos1949theorem}, it follows that
\color{black}



\begin{equation}
\lim_{n \to \infty} v_n =   \frac{1}{\E \min (Z,W)} >0~. \label{convergence vn}
\end{equation}
Thus, for some constant $c>0$, we have 
\begin{equation}
\E J_n \ge c \sum_{m=1}^n p_m \to \infty \quad \text{as } n\to \infty~. \label{lowerbound}
\end{equation}
\end{proof}

\noindent
We use the second-moment method to show that $J_n \to \infty$ in probability.

\begin{lemma}\label{Lemma second moment}
Assume that $\E Z = \infty$ and $\E \min (Z,W) < \infty$, and let $\V J_n$ denote the variance of the random variable $J_n$. Then
\begin{equation}
\frac{\V J_n}{(\E J_n)^2}\to 0 \quad \text{as } n \to \infty~. \label{scnd moment method}
\end{equation}
In particular, we have:
$$\PP\left(J_n\le \frac{\E J_n}{2}\right)=\PP\left(J_n - \E J_n \le - \frac{\E J_n}{2}\right)\le 4\frac{\V J_n}{(\E J_n)^2}\to 0~,$$
and therefore $J_n \to \infty$ in probability as $n \to \infty$.
\end{lemma}

\begin{proof}[Proof of Lemma \ref{Lemma second moment}]
To show \eqref{scnd moment method}, we prove that $\E[J_n^2]/(\E J_n)^2\to 1$ as $n \to \infty$. We first define
$$
p_n^*=\PP(\max(Z,W)\ge n) = 2p_n - p_n^2~.
$$
Recall that $v_n = \PP(0 \in S_n)$, as introduced earlier. Then, for $n\ge 1$,
\begin{align}
    (\E J_n)^2 &= 2\sum_{1\le k<m\le n} p_m^* p_k^* v_{n-m}v_{n-k} + \sum_{1\le m \le n} (p_m^*)^2 v_{n-m}^2 \nonumber \\
    &= 2\sum_{1\le k<m\le n} p_m^* p_k^* v_{n-m}v_{n-k} + O(1)
\end{align}
since $\sum_{1\le m \le n} (p_m^*)^2 v_{n-m}^2$ is bounded by $4\E
\min (Z,W) <\infty$.

Now, define for any $m\ge 1$ the random variable $X_m=\max(Z_m, W_m)$ and observe that for any $1\le k < m \le n$, the events $\left\{ X_k\ge k\right\}$ and $\left\{ X_m\ge m, m\in S_n, k \in S_n\right\}$ are independent.
Thus, for any fixed $n\ge 1$, we may write
\begin{align}
    \E\left[ J_n^2\right] &= 2\sum_{1\le k < m \le n} \prob{m\in S_n, k\in S_n, X_m\ge m, X_k\ge k} + \E J_n \nonumber \\
    &= 2\sum_{1\le k < m \le n} p_k^*\prob{m\in S_n, k\in S_n, X_m\ge m} + \E J_n \nonumber \\
    &= 2\sum_{1\le k < m \le n} p_k^* \procond{k\in S_n, X_m \ge m}{m\in S_n}\prob{m\in S_n}+ \E J_n \nonumber 
\end{align}
If $m\in S_n$, the event $\{k\in S_n, X_m\ge m\}$ can only happen if one of the two random variables $Z_m$ and $W_m$ is greater or equal to $m$ and the other one is smaller than $m-k$, for otherwise $k$ could not belong to $S_n$ when $1\le k<m \le n$. Thus,
using the independence of $Z_m$ with respect to $W_m$, $\{k\in S_n\}$ and $\{m\in S_n\}$, and using the fact that 
$p_m^* + p_m^2 = 2p_m$,
\begin{align}
     \E&\left[ J_n^2\right] \\
     &= 2\sum_{1\le k < m \le n} p_k^* \times 2 \procond{Z_m\ge m, W_m\le m-k , k\in S_n}{m\in S_n}\prob{m\in S_n}+ \E J_n \nonumber \\
     &= 2\sum_{1\le k < m \le n} p_k^* (p_m^*+p_m^2)v_{n-m}\procond{W_m\le m-k , k\in S_n}{m\in S_n}+ \E J_n \nonumber \\
    &= 4\sum_{1\le k < m \le n} p_k^* (p_m^*+p_m^2)v_{n-m}\sum_{i=1}^{k-m}\prob{k\in S_{m-i}}\prob{W_m=i}+ \E J_n \nonumber \\
    &= 2\sum_{1\le k < m \le n} p_k^* (p_m^*+p_m^2)v_{n-m}\sum_{i=1}^{m-k}v_{m-k - i}q_i+ \E J_n ~.
\end{align}
Observe that
\begin{align}
    \sum_{1 \le k < m \le n}p_k^*p_m^2v_{n-m}\sum_{i=1}^{m-k}q_iv_{m-k-i} &\le \sum_{1 \le k < m \le n}p_k^*p_m^2v_{n-m} \nonumber \\
    &\le  \sum_{1 \le k < m \le n}p_k^*p_m^2 \nonumber \\
    &\le \left( \sum_{m=1}^\infty p_m ^2\right) \left(\sum_{k=1}^n p_k^*\right) \nonumber \\
    &=  O\left( \E J_n\right) \nonumber
\end{align}     
by using the inequality in \eqref{lowerbound}, which leads to 
\begin{equation}
    \E[J_n^2] = 2\sum_{1\le k < m \le n} p_k^*p_m^*v_{n-m}\sum_{i=1}^{m-k}v_{m-k-i}q_i+ O\left( \E J_n \right)~.
\end{equation}
Fix some $\epsilon >0$. By \eqref{convergence vn}, we know that there exists some constant $x>0$ such that for all $m \ge x$, we have $|v_m-\lambda|<\epsilon$,  where $\lambda = 1/(\E\min(Z,W))$.

\noindent If $m\ge x$, see that
\begin{align*}
    \sum_{i=1}^m q_i v_{m-i} &\le \sum_{i=1}^{m-x} q_i (\lambda + \epsilon) + \sum_{m-x<i\le m} q_i \\
    &\le (\lambda + \epsilon) + p_{m-x}~.
\end{align*} 
Thus, there exists some constant $y>x$ such that for all $m \ge y$,
$$ \sum_{i=1}^m q_i v_{m-i}\le \lambda + 2 \epsilon~.$$
Moreover,
\begin{align*}
    \sum_{\substack{1\le k < m \le n \\ \textit{s.t. } m-k < y} }p_m^*v_{n-m}\underbrace{p_k^*\sum_{i=1}^{m-k}v_{m-k-i}q_i}_{\le 1}  
    &\le \sum_{1\le m \le n}p_m^* v_{n-m} \sum_{k=m-y}^m 1 \\
    &= (y+1) \E J_n = O\left( \E J_n\right)~,
\end{align*}
and 
\begin{align*}
    \sum_{\substack{1\le k < m \le n \\ \textit{s.t. } n-m < y} }p_k^*\underbrace{p_m^*v_{n-m}\sum_{i=1}^{m-k}v_{m-k-i}q_i}_{\le 1} &\le \sum_{n-y\le m \le n}1 \sum_{k=1}^n p_k^* \\
    &\le (y+1) \sum_{k=1}^n p_k^*
    = O\left( \E J_n\right)~.
\end{align*}
Putting everything together, we have that 
\begin{align}
    \E [J_n^2] = \sum_{m=1}^{n-y}\sum_{k=1}^{m-y} p_m^*p_k^*v_{n-k}\sum_{i=1}^{m-k} q_i v_{m-k-i} + O\left( \E J_n\right)~. \label{equivalent 1}
\end{align}
Similarly, we have
\begin{align}
    (\E J_n)^2 = \sum_{s=1}^{n-y}\sum_{k=1}^{m-y} p_m^*p_k^*v_{n-m}v_{n-k} + O\left( \E J_n\right)~. \label{equivalent 2}
\end{align}

\noindent Finally, observe that for any $n> y$ we have:
\begin{align}
\sum_{m=1}^{n-y}\sum_{k=1}^{m-y} p_m^*p_k^*\underbrace{v_{n-m}\sum_{i=1}^{m-k} q_i v_{m-k-i}}_{\le (\lambda + 2\epsilon)^2} \le (\lambda +2\epsilon)^2\sum_{m=1}^{n-y}\sum_{k=1}^{m-y} p_m^*p_k^* \label{upper J squared}
\end{align}
and
\begin{align}
    \sum_{m=1}^{n-y}\sum_{k=1}^{m-y} p_m^*p_k^*\underbrace{v_{n-m}v_{n-k}}_{\ge (\lambda - \epsilon)^2} \ge (\lambda - \epsilon)^2\sum_{m=1}^{n-y}\sum_{k=1}^{m-y} p_m^*p_k^*~. \label{lower J squared}
\end{align}
From \eqref{equivalent 1}, \eqref{equivalent 2}, \eqref{upper J squared} and \eqref{lower J squared} it follows that
$$\limsup_{n\to \infty} \frac{\E[J_n^2]}{(\E J_n)^2}\le \left( \frac{\lambda + 2\epsilon}{\lambda - \epsilon}\right)^2 ~.$$
Since this is true for any $\epsilon>0$ small enough, and $\E[J_n^2]\ge (\E J_n)^2$, we conclude that $$\frac{\E [J_n^2]}{(\E J_n)^2}\to 1$$
as $n \to \infty$.\end{proof}

\color{black}
We now turn to the proof of Proposition~\ref{weak behavior case 2}, which asserts that the number of leaves in the medium-tail scenario diverges in probability. We have already shown that $J_n \to \infty$ in probability, where $J_n$ counts the number of edges connecting vertices in the chain $S_n$ to leaves via their longest outgoing edge. Proposition~\ref{weak behavior case 2} follows from this result, but some care is needed: $J_n$ counts edges, not distinct leaves, and multiple edges may connect to the same leaf. The key idea of the proof is to show that the number of edges of the type $\max(Z_n,W_n)$ connecting to any single leaf cannot grow to infinity. This ensures that the divergence of $J_n$ indeed implies that the number of distinct leaves connected to $T_n$ also diverges to infinity in probability.
\color{black}

\begin{proof}[Proof of Proposition \ref{weak behavior case 2}]
 Let us define $$\Lset_n^* =\left\{ k\le 0 : \textit{there exists an }  m\in S_n \textit{ such that } \max (Z_m,W_m)= m-k\right\}~.$$ In other words, $\Lset_n^*$ is the subset of vertices in $\Lset_n$ that are connected to some vertex in $S_n$ with positive index through the longest edge (i.e., given by $\max (Z,W)$).

Given $n\ge 1$ and $k\in \Z$, we introduce the 
random variable $$D_{n,k}=\sum_{m=1}^n \IND_{\max(Z_m,W_m)=m-k}.$$
Note that for any fixed $x>0$,  
\begin{align}
\PP&\left(\left|\Lset_n^*\right|<x\right) \\
&\le \PP\left( \left[\max_{k\in \Lset_n^*} \sum_{m\in S_n}\IND_{\max(Z_m,W_m)=m-k}< \frac{\E J_n}{2x} \right] \cap \left[\left|\Lset_n^*\right|<x  \right]\right) \nonumber\\
&\qquad +\PP \left(\cup_{k\le 0} \left[ D_{n,k} \ge \frac{\E J_n}{2x} \right] \right)~. \label{split sum}
\end{align}

Since $J_n$ can be written as $$J_n=\sum_{k\in\Lset_n^*}\sum_{m\in S_n}\IND_{\max(Z_m,W_m)=m-k}~, $$
we have that 
\begin{align}
&\PP\left( \left[\max_{k\in \Lset_n^*} \sum_{m\in S_n}\IND_{\max(Z_m,W_m)=m-k}< \frac{\E J_n}{2x} \right] \cap \left[\left|\Lset_n^*\right|<x  \right]\right) \nonumber\\
&\quad\le \PP\left(J_n\le \frac{\E J_n}{2}\right) \nonumber \\
&\quad\le 4\frac{\V J_n}{(\E J_n)^2}\to 0 ~. \label{term 1}
\end{align}
Moreover, for any $k\le 0$ and $y>0$ , by Chernoff's bound \cite{chernoff}, \cite{boucheron}, we have that
\begin{align*}
\PP(D_{n,k} > y) &= \PP\left(\sum_{m=1}^n \IND_{\max(Z_m,W_m)=m-k}> y\right) \\
&\le \exp\left( y-p_{1-k}^* - y\log\left(\frac{y}{p_{1-k}^*} \right) \right) \\
&\le \left( \frac{e}{y}\right)^y (p_{1-k}^*)^y ~.
\end{align*}
Let $y_n=\E J_n/(2x)$ and assume that $n$ is large enough so we have $y_n \ge 2$ and $$\sum_{m=1}^\infty (p_{m}^*)^{y_n} \le \sum_{m=1}^\infty (p_{m}^*)^2<\infty \text{ (since } \E\min(Z,W)<\infty \text{ )}~.$$
Thus, by the union bound, 
\begin{align}
    \PP \left(\cup_{k\le 0} \left[ D_{n,k} \ge \frac{\E J_n}{2x} \right]\right) 
    &\le \sum_{k\le 0} \left( \frac{e}{y_n}\right)^{y_n} (p_{1-k}^*)^{y_n} \nonumber \\
    &\le \left( \frac{e}{y_n}\right)^{y_n}  \sum_{k=1}^\infty (p_{k}^*)^2 \to 0 ~. \label{term 2}
\end{align}
By putting \eqref{split sum}, \eqref{term 1}, and \eqref{term 2} together we have that $\PP\left(\left|\Lset_n^*\right|<x\right)\to 0$ as $n\to \infty$ for every fixed $x>0$, which proves Proposition \ref{weak behavior case 2}. 
\end{proof}

\color{black}
Having established divergence in probability, we are now ready to prove almost sure divergence.
\color{black}

\begin{proof}[Proof of Theorem \ref{ L to infinity almost surely}]
    For any $n\ge 1$, define $Y_n=\min\{ m-\min(Z_m,W_m) : m\ge n\}$. A simple union bound gives, for arbitrary $k$:
    $$\prob{Y_n\le k}\le \sum_{m\ge n} \prob{\min(Z_m,W_m)>m-k}.$$
    Since $\E \min (Z,W) < \infty$, the sum above is finite and goes to $0$ as $n$ goes to infinity. Hence, $Y_n\to \infty$ in probability.

    Let us fix $n\ge 0$ and apply Corollary \ref{corollary surf} to
    the model obtained by shifting by $n$ the set of integers so that
    leaves are indexed by the set $\{k: k\le n\}$ and the edges are
    given by the sequences of random variables $(Z_k)_{k\ge n+1}$ and
    $(W_k)_{k\ge n+1}$. This gives us a random variable $N_n> n$,
    finite almost surely, such that
    $N_n\in S_k$ for all $k \ge N_n$. Let $I_n \le n\}$
    be the root of this tree, and note that
    $$Y_n\le I_n \le n <  N_n~.$$

    \begin{figure}[H]
    \centering
    \includegraphics[scale=.5]{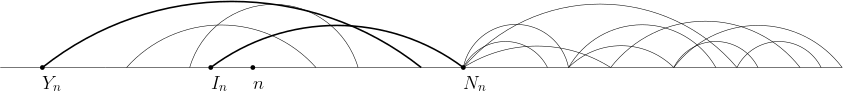}
    \caption{Illustration of the proof of Theorem \ref{ L to infinity almost surely}}
\end{figure}

    Since $Y_n \to \infty$ in probability,  $I_n \to \infty$ in probability as well. Moreover, for all $n\ge 1$, one can choose $N_n$ such that $N_n-n$ is distributed as $N_0$.
    The fact that all vertices $k\ge N_n$ belong to the tree rooted at $I_n$ in this single-choice model implies that $I_n\in T_k$ for all $k\ge N_n$, and in particular $\leaves_{I_n}\le \leaves_k$ for all $k\ge N_n$.
    From Proposition \ref{weak behavior case 2} we know that $\leaves_n\to \infty$ in probability. Hence, $\leaves_{I_n}\to \infty$ in probability.
    Moreover, for any fixed $m\ge 1$ we have
    \begin{align*}
        \prob{\cup_{k \ge 2n} \left[ \leaves_k\le m \right] } 
        &\le \prob{N_n > 2n} + \prob{N_n \le 2n, \leaves_{I_n}\le m} \\
        &\le \prob{N_0 > n} + \prob{\leaves_{I_n}\le m}~.
    \end{align*}
    Thus, for any $m\ge 1$,  $\prob{\cup_{ k \ge 2n} [\leaves_k\le m] }\to 0$ as $n\to \infty$. By continuity of measure, this implies that for any $m \ge 1$,
    $$\prob{\cap_{n\ge 1} \cup_{k \ge 2n} \left[ \leaves_k \le m\right]}=0 .$$
    By the union bound, $$\prob{\cup_{m \ge 1} \cap_{n\ge 1} \cup_{k \ge 2n } \left[\leaves_k\le m\right] }=0~,$$
    which implies that $\leaves_n\to \infty$ almost surely.
\end{proof}

\subsection{Heavy tails}
\label{sec:heavy}

In this section, we study the heavy-tailed case, that is, $\E \min(Z,W) = \infty$. From Theorem \ref{a.s. thm Emin=infty} we already know that $\leaves_n$ cannot go to infinity with probability one. 
Here we show that, in contrast to this, $\leaves_n\to \infty$ in probability when $\sum_{n=1}^\infty r_n^2 <\infty$.

\begin{lemma} \label{r-bound}
    For any integers $k \in \Z$ and $n\ge \max(1,k)$, $$\prob{k\in T_n^Z}\le r_{n-k}~.$$
\end{lemma}

\begin{proof}
    Clearly, we have $\prob{k\in T_n^Z}=r_{n-k}$ whenever $k\ge 0$. Moreover, $r_n$ satisfies the recursion $$r_n=\sum_{i=1}^n q_i r_{n-i}~.$$
    Thus, for any $k>0$, we have
    \begin{align*}
        \prob{-k\in T_n^Z} &= \sum_{i=1}^n \prob{Z_i=i+k}\prob{i\in T_n} \\
        &= \sum_{i=1}^n q_{k+i}r_{n-i}\\
        &= \sum_{i=k+1}^{n+k} q_{i}r_{n+k-i} \\
        &\le  \sum_{i=1}^{n+k} q_{i}r_{n+k-i} = r_{n+k}~.
    \end{align*}
\end{proof}
\begin{proof}[Proof of Theorem \ref{weak thm Emin=infty}]
    Assume that $\sum_{n=1}^\infty r_n^2 <\infty$. Fix a constant $w>0$ and let $I_n^w$ denote the number of intervals of the form $(kw,(k+1)w]$, with $k\in \Z_{\ge 0}$, that are intersected by the chain $T_n^Z$. Note that, for any $x>0$, if $I_n^w < x$ then $T_n^Z < wx $, thus 
    \begin{align}
        \prob{I_n^w<x} &\le \prob{T_n^Z < wx} \nonumber \\
        &\le \prob{Z_1+\dots + Z_{\lfloor wx \rfloor}>n}\to 0 ~. \label{I_n goes to infty}
    \end{align}
    Let $\ell \ge 1$ be an arbitrary integer and observe that, if the event $I_n^w \ge 2\ell$ occurs, then one can define a random set of vertices $A_n$ such that
    \begin{enumerate}
        \item $A_n$ is independent of $(W_k)_{k\ge 1}~,$
        \item $A_n\subset T_n^Z~,$
        \item $|A_n|=\ell~,$
        \item $\text{for all } k\neq m \in A_n~,~|k-m|\ge w~.$
    \end{enumerate}
    Conditionally on $I_n^w\ge 2\ell$, we have
    \begin{align}
        &\prob{\cup_{k\neq m, k \in A_n}  T_k^W\cap T_m^W\neq \emptyset} \\
        &\quad\le \E \sum_{\substack{k,m\in A_n \\ \textit{s.t. } k\neq m} } \IND_{T_k^W\cap T_m^W \neq \emptyset} \nonumber \\
        &\quad\le \E \sum_{\substack{k,m\in A_n \\ \textit{s.t. } k\neq m} } \prob{T_k^W\cap T_m^W \neq \emptyset} \nonumber \\
        &\quad\le {\ell\choose 2} \max_{\substack{k,m\le n\\ \textit{s.t. } |k-m|\ge w} }\prob{T_k^W\cap T_m^W \neq \emptyset} \label{collision existence bound}
    \end{align}
    by using the first moment method and the independence of $A_n$ with respect to $(W_k)_{k\ge 1}$. Let $k<m$ be a pair of integers such that $|k-m|\ge w$ and observe that $$\prob{T_k^W\cap T_m^W \neq \emptyset}=\prob{T_k^W\cap T_m^Z \neq \emptyset}~.$$
    Thus, the union bound, Lemma \ref{r-bound} and the Cauchy-Schwarz inequality yield
    \begin{align}
        \prob{T_k^W\cap T_m^W \neq \emptyset} &\le \sum_{i=-\infty}^k \prob{i\in T_k^W}\prob{i\in T_m^Z} \nonumber \\
        &\le \sum_{i=-\infty}^k r_{k-i}r_{m-i}  \nonumber \\
        &\le \sqrt{\sum_{i=0}^\infty r_i^2} \times\sqrt{\sum_{i=m-k}^\infty r_i^2}. \label{collision bound}
    \end{align}
    From \eqref{collision existence bound} and \eqref{collision bound}, it follows that 
\begin{equation}
    \prob{\cup_{k\neq m, k \in A_n}  T_k^W\cap T_m^W\neq \emptyset} \le {\ell\choose 2}\sqrt{\sum_{i=0}^\infty r_i^2} \times\sqrt{\sum_{i=w}^\infty r_i^2}.\label{final collision existence bound}
    \end{equation}
    Fix $\epsilon > 0 $. From \eqref{final collision existence bound} we know that we can choose $w$ big enough so we have  for $n\ge 1$,
    \begin{equation}
        \prob{\cup_{k\neq m, k \in A_n}  T_k^W\cap T_m^W\neq \emptyset}\le \epsilon.
    \end{equation}
    Finally, observe that if $\leaves_n < \ell$ and $I_n^w\ge 2\ell$, then there have to be at least two distinct vertices $k,m$ in $A_n$ such that the chains $T_k^W$ and $T_m^W$ meet each other: otherwise, each chain $T_k^W$ with $k\in A_n$ would lead to a distinct leaf in $T_n$, contradicting the event $\leaves_n<\ell$. Thus, by \eqref{final collision existence bound} we have
    $$\prob{\leaves_n<\ell} \le \prob{I_n^w < 2\ell} + \epsilon~.$$
    From \eqref{I_n goes to infty} it follows that for any $\epsilon >0 $ and any integer $\ell\ge 1$
    $$\limsup_{n\to \infty} \prob{\leaves_n<\ell} \le \epsilon~,$$
    which concludes the proof of Theorem \ref{weak thm Emin=infty}.
    \end{proof}

\section{Generalization to $k$ choices.}
\color{black}
A natural generalization of the model allows for $k \ge 2$ choices. Fix $k\ge 2$ and consider $k$ independent sequences of i.i.d.\ random variables $(Z_i^{(1)})_{i \ge 1}, \dots , (Z_i^{(k)})_{i \ge 1}$. Given a sequence of preference values $(U_i)_{i \le 0}$, we define a sequential random coloring by setting $V_n = U_n$ and $C_n = n$ for all $n \le 0$, and for $n > 0$ letting
\[
C_n = C_{n-Z_n^{(j)}} \quad \text{and} \quad V_n = V_{n-Z_n^{(j)}} \quad \text{where} \quad j = \argmin_{1 \le i \le k} V_{n-Z_n^{(i)}}
\]
for $n \ge 1~.$
The set of vertices reached from vertex $n$ is now defined recursively as  
\[
T_n = \{n\} \cup \left(\bigcup_{i=1}^k T_{n-Z_n^{(i)}} \right), \quad n \ge 1~,
\]
with $T_n = \{n\}$ for all $n \le 0$. As in the two-choice model, we may then define $\rightmost_n$, $\leaves_n$, and $\leftmost_n$ accordingly. The graph associated with this model therefore, contains multiple embedded copies of the two-choice process.  It follows from this observation that consistency in the two-choice model carries over to the $k$-choice model for any $k \ge 2$, and for each of the three configurations of the initial pool of topics considered in this paper. As a consequence, Theorem \ref{M diverges a.s. } extends automatically to the $k$-choice model: 
\begin{theorem}
    Assume that $q_i>0$ for all $i\in \N$ and $\E Z=\infty$. Then, in the $k$-choice model, with
    probability one, $$\leftmost_n\to -\infty~.$$
Thus, the system is strongly consistent in configuration (ii).
\end{theorem}

However, the increase in the number of available recommendations allows one to work under weaker assumptions on the tail of $Z$, and in this setting the results can be sharpened. 
In particular, one can adapt the proof given in Section~\ref{sec:moderate} by modifying the recursion defining $S_n$ to
\[
S_n = \{n\} \cup S_{\,n-\min_{1\le i \le k} Z_n^{(i)}},
\]
and by controlling the growth of 
\[
J_n = \sum_{m=1}^n \IND_{m \in S_n}\,\IND_{\max_{1\le i \le k} Z_m^{(i)} \ge m}~,
\]
to establish the following result.  

\begin{theorem}
    Assume that the distribution of $Z$ is aperiodic, and that 
    \[
    \E[Z] = \infty, \qquad \E\!\left[\min_{1\le i \le k} Z^{(i)}\right] < \infty~,
    \]
    where $Z^{(1)}, \dots, Z^{(k)}$ are independent copies of $Z$.  
    Then, in the $k$-choice model,  
    \[
    \leaves_n \to \infty \quad \text{almost surely.}
    \]
\end{theorem}

The arguments used in the proofs of Theorems \ref{M bounded as}, \ref{a.s. thm Emin=infty} and Theorem \ref{a.s thm Emin<infty} extend naturally to the present setting, leading to the following three results.

\begin{theorem}
Let $\leftmost_\infty=\inf_{n\in \N} \leftmost_n$. If $\E Z < \infty$, then $\leftmost_\infty$ is finite almost surely.
Thus, in the $k$-choice model, the sequences $(\rightmost_n)_{n\in\N},(\leftmost_n)_{n\in\N}$ and $(\leaves_n)_{n\in\N}$ are almost surely bounded.
\end{theorem}

\begin{theorem}
    Let $Z^{(1)},\dots, Z^{(k)}$ be independent copies of $Z$. Assume that $$\E \min_{1\le i \le k} Z^{(i)} = \infty~.$$ Then, in the $k$-choice model, with probability one,
  $$
  \liminf_{n\to \infty} \leaves_n < \infty~.
  $$
  \end{theorem}

\begin{theorem}
    Let $Z^{(1)},\dots, Z^{(k)}$ be independent copies of $Z$. 
    \begin{enumerate}
        \item If $\E \min_{1\le i \le k} Z^{(i)}<\infty$, then $\sup_{n\in\N} |\rightmost_n| < \infty$  almost surely in the $k$-choice model.
        \item If $\E \min_{1\le i \le k} Z^{(i)} = \infty$, then $\sup_{n\in\N} |\rightmost_n| = \infty$ almost surely in the $k$-choice model.
    \end{enumerate}
\end{theorem}

In summary, the consistency table from Section \ref{sec:summary} remains unchanged in the $k$-choice model. The only difference is that the moderate-tail regime is now characterized by the conditions $\E \min_{1 \le i \le k} Z^{(i)} < \infty$ and $\E Z^{(1)} = \infty$, while the heavy-tail regime corresponds to $\E \min_{1 \le i \le k} Z^{(i)} = \infty$.

\color{black}

\section{Conclusion}

We introduced and studied a simple mathematical model for
recommendation systems based on giving each user two random choices,
which include a mixture of past choices and untried options.  In this
setup, we identify three regimes based on the tail behavior of the
sizes of the jumps in the past.
\color{black}
Although our model is primarily theoretical, its analysis offers qualitative insights that may inform the design of recommendation systems. In particular, it highlights how structural properties of the recommendation process can influence long-term behavior and diversity. Extending the model to incorporate personalization and adaptive feedback remains an interesting direction for future work.

\subsection{Comparison with standard baselines}

To situate our results within the broader landscape of recommendation methods, it is natural to consider connections to standard algorithmic baselines. Classical approaches, such as collaborative filtering, multi-armed bandits, and deep learning–based recommender systems, rely on adaptive strategies to cope with noisy or partially observed user preferences. While our model is highly stylized and focuses on the structural properties of the recommendation process rather than optimization or learning, comparing its behavior conceptually to these frameworks helps highlight both its scope and its limitations.

A natural point of comparison is with multi-armed bandit algorithms \cite{LaSz20, BuCe12}, which also model sequential decision-making and recommendation dynamics. However, the two frameworks differ fundamentally: in our model, the “rewards’’ (preference values) are static and noise-free, with randomness arising solely from the distribution of edge lengths, while in bandit algorithms, the primary challenge is to manage uncertainty in the reward distributions through exploration and exploitation. Because of this distinction, standard bandit baselines are not directly applicable to our setting, and a meaningful comparison would require a reformulation of the bandit problem that departs significantly from its usual assumptions.

A closer comparison can be drawn with bandit algorithms in the noiseless setting, where the reward associated with each arm is fixed and known after a single pull. In that framework, the problem is essentially trivial: the optimal arm is identified immediately, and the learning phase vanishes. By contrast, our model remains nontrivial even in the absence of noise, since the randomness lies not in the evaluation of topics but in the structure of the recommendation process itself. The stochastic nature of the edge lengths forces the system to repeatedly revisit or discover new topics, giving rise to rich asymptotic behaviors such as strong or weak consistency.

\color{black}

\subsection{Perspectives and open problems}
\color{black}
One might naturally expect the $k$-choice model to exhibit distinct behavior in the $k+1$ regimes defined by the tail conditions
\[
\E \min_{1\le i \le \ell} Z^{(i)} = \infty \quad \text{and} \quad \E \min_{1\le i \le \ell+1} Z^{(i)} < \infty,
\]
for each $\ell \in \{0, \dots, k\}$ (with the understanding that the left-hand condition is omitted when $\ell = 0$ and the right-hand condition is omitted when $\ell = k$).  
Although our results extend straightforwardly to the $k$-choice setting, the current consistency analysis does not distinguish between these individual regimes. A deeper understanding of the distinctions between these regimes and their implications for recommendation outcomes would be of significant interest.

The present model remains highly abstract and does not directly reflect real-world applications. For instance, the assumption of static preference values for each topic may not capture the evolution of actual user behavior. A natural direction for future work is to develop dynamic generalizations of the model that account for time-varying preferences for each topic.

\color{black}


\newpage
\section*{Appendix}

\begin{table}[H]
\centering
\renewcommand{\arraystretch}{1} 
\begin{tabular}{ll}
\hline
\textbf{Symbol} & \textbf{Description} \\
\hline
$T_n$ & Set of vertices reached by vertex $n$. \\
$\Lset_n$ &\vtop{\hbox{\strut Set of vertices with non-positive index (or \textit{leaves}) that is reached}\hbox{\strut from vertex $n$. ($\Lset_n=T_n\cap \{0, -1, -2, \dots\})$}}\\
$\leaves_n$ & Number of leaves reached from $n$. ($\leaves_n = |\Lset_n|$)\\
$\rightmost_n$ & Rightmost leaf in $T_n$. $(\rightmost_n=\max \left\{ i : i \in  \Lset_n \right\})$\\
$\leftmost_n$ &Leftmost leaf in $T_n$. ($\leftmost_n=\min \left\{ i : i \in  \Lset_n \right\}$)\\
$T_n^Z$ (resp. $T_n^W)$ & \vtop{\hbox{\strut Chain in $T_n$ obtained by repeatedly taking}\hbox{\strut the $Z$-edges (resp. the $W$-edges)}}\\
$S_n$ & \vtop{\hbox{\strut Chain in $T_n$ obtained by repeatedly taking shortest}\hbox{\strut outgoing edges}}\\
$J_n$ & \vtop{\hbox{\strut Number of vertices in $S_n$ connected to a leaf via their}\hbox{\strut longest outgoing edge}}\\
$C_n$ & Topic recommended at time $n$ \\
$Z_n, W_n$ & Independent random delays for recommendations \\
$U_i$ & Preference value of topic $i$ ($i\le 0$) \\
$V_n$ & Preference value of topic seen at time $n$\\
$p_i, q_i$ & Probabilities associated with $Z$ ($p_i = \PP(Z \ge i)$, $q_i = \PP(Z=i)$) \\
$p_i^*$ & \vtop{\hbox{\strut Upper tail probability associated with $\max(Z,W)$}\hbox{\strut ($p_i^*=\PP(\max(Z,W)\ge i) = 2p_i - p_i^2$)}}\\
$r_n$ & \vtop{\hbox{\strut Probability that vertex $0$ belongs to the $Z$-chain starting}\hbox{\strut at vertex $n$ ($r_n =\prob{0\in T_n^Z}=\prob{0\in T_n^W})$}}\\
$v_n$ & \vtop{\hbox{\strut Probability that vertex $0$ belongs to the chain $S_n$}\hbox{\strut ($v_n =\prob{0\in S_n}$)}}\\
$\IND_A$ & Indicator function of event $A$ \\
$\E X$ & Expected value of random variable $X$\\
$\V X$ & Variance of random variable $X$\\

\hline
\end{tabular}
\caption*{Notation Table}
\label{tab:notation}
\end{table}

\section*{Funding}
Luc Devroye was supported by NSERC
under grant number RGPIN-2024-04164.
G\'abor Lugosi acknowledges the support of Ayudas Fundación BBVA a
Proyectos de Investigación Científica 2021 and
the Spanish Ministry of Economy and Competitiveness grant PID2022-138268NB-I00, financed by MCIN AEI 10.13039 501100011033,
FSE+MTM2015-67304-P, and FEDER, EU.


\end{document}